\documentclass[12pt]{article}
\usepackage{amscd, amsmath, amssymb, amsfonts, amsthm}
\usepackage{euscript}
\usepackage{graphicx}
\usepackage[square, comma, sort&compress, numbers]{natbib}
\usepackage[colorlinks,linkcolor=blue,citecolor=magenta,anchorcolor=red, bookmarksopen=false]{hyperref}

\def\pref#1{(\ref{#1})}

\theoremstyle{plain}
\newtheorem{prop}{Proposition}[section]
\newtheorem{thm}[prop]{Theorem}
\newtheorem{lem}[prop]{Lemma}
\newtheorem{cor}[prop]{Corollary}

\newtheorem{conj}[prop]{Conjecture}

\theoremstyle{definition}

\newtheorem{defn}[prop]{Definition}
\newtheorem{rem}[prop]{Remark}

\makeatletter
\def\p@figure{Fig. }
\def\p@enumii{}

\makeatother

\newcommand{\V}{\mathrm{V}}
\newcommand{\E}{\mathrm{E}}
\newcommand{\N}{\mathrm{N}}
\newcommand{\T}{\mathrm{Tr}}

\renewcommand{\b}[1]{\overline{#1}}

\newcommand{\se}{\subseteq}

\newcommand{\link}{\mathrm{link}}
\newcommand{\sm}{\setminus }

\newcommand{\give}{$\Rightarrow$}
\newcommand{\rgive}{$\Leftarrow$}
\newcommand{\ifof}{if and only if }

\newcommand{\f}[2]{\frac{#1}{#2}}

\newcommand{\tohi}{\emptyset}

\newcommand{\Q}{\mathbb{Q}}

\begin{document}

\title{Trung's Construction and the Charney-Davis Conjecture}
\author{Ashkan Nikseresht and Mohammad Reza Oboudi\\
\it\small Department of Mathematics, Shiraz University,\\
\it\small 71457-13565, Shiraz, Iran\\
\it\small E-mail: ashkan\_nikseresht@yahoo.com\\
\it\small E-mail: mr\_oboudi@yahoo.com }
\date{}
\maketitle

\begin{abstract}
We consider a construction by which we obtain a simple graph $\T(H,v)$ from a simple graph $H$ and a non-isolated
vertex $v$ of $H$. We call this construction ``Trung's construction''. We prove that $\T(H,v)$ is well-covered, W$_2$
or Gorenstein if and only if $H$ is so. Also we present a formula for computing the independence polynomial of
$\T(H,v)$ and investigate when $\T(H,v)$ satisfies the Charney-Davis conjecture. As a consequence of our results, we
show that every Gorenstein planar graph with girth at least four, satisfies the Charney-Davis conjecture.
\end{abstract}


                                        \section{Introduction}

Throughout this paper, $K$ is a field, $S=K[x_1,\ldots, x_n]$ and $G$ denotes a simple undirected graph with vertex set
$\V(G)=\{v_1,\ldots, v_n\}$ and edge set $\E(G)$. Recall that the \emph{edge ideal} $I(G)$ of $G$ is the ideal of $S$
generated by $\{x_ix_j|v_iv_j\in \E(G)\}$. Many researchers have studied how algebraic properties of $S/I(G)$ relates
to combinatorial properties of $G$ (see \cite{large girth, planar goren, stanley, hibi} and references therein). Recall
that $G$ is called a Gorenstein (resp, Cohen-Macaulay or CM for short) graph over $K$, if $S/I(G)$ is a Gorenstein
(resp. CM) ring. When $G$ is Gorenstein (resp. CM) over every field, we say that $G$ is Gorenstein (resp. CM). Finding
combinatorial conditions on a graph equivalent to being Gorenstein has recently gained attention. For example, in
\cite{large girth} a characterization of planar Gorenstein graphs of girth at least four is presented. Also in
\cite{planar goren} a condition on a planar graph equivalent to being Gorenstein is stated.

An importance of characterization of Gorenstein graphs comes from the Charney-Davis conjecture on the Euler
characteristic of certain manifolds (see \cite{charney} and \cite{frontier}). This conjecture could be restated in
terms of independence polynomials of Gorenstein graphs (see \pref{char conj}).

In this paper, first we recall some needed concepts and preliminary results. Then in Section 3, we show that planar
Gorenstein graphs with girth at least four satisfy the Charney-Davis conjecture. All Gorenstein graphs with girth four
are constructed using a recursive construction. We call a more general form of this recursive construction ``Trung's
construction'' and show that this construction preserves several properties related to independent sets such as being
well-covered, W$_2$ or Gorenstein. We also present a formula for computing the independence polynomial of graphs
constructed using Trung's construction and study when these graphs satisfy the Charney-Davis conjecture.
                                        \section{Preliminaries}

Recall that a \emph{simplicial complex} $\Delta$ on the vertex set $V=\{v_1,\ldots, v_n\}$ is a family of subsets of
$V$ (called \emph{faces}) with the property that $\{v_i\}\in \Delta$ for each $i\in [n]=\{1,\ldots,n\}$ and if $A\se
B\in \Delta$, then $A\in \Delta$. In the sequel, $\Delta$ always denotes simplicial complex. Thus the family
$\Delta(G)$ of all cliques of a graph $G$ is a simplicial complex called the \emph{clique complex} of $G$. Also
$\Delta(\b G)$ is called the \emph{independence complex} of $G$, where $\b G$ denotes the complement of $G$. Note that
the elements of $\Delta(\b G)$ are independent sets of $G$. If $\Delta=\Delta(\b G)$ for some graph $G$, then $\Delta$
is called a \emph{flag} complex. The ideal of $S$ generated by $\{\prod_{v_i\in F} x_i|F\se V$ is a non-face of
$\Delta\}$ is called the \emph{Stanley-Reisner ideal} of $\Delta$ and is denoted by $I_\Delta$ and $S/I_\Delta$ is
called the \emph{Stanley-Reisner algebra} of $\Delta$ over $K$. Therefore we have $I_{\Delta(\b G)}= I(G)$. Many
researchers have studied the relation between combinatorial properties of $\Delta$ and algebraic properties of
$S/I_\Delta$, see for example \cite{hibi, stanley, our chordal, my vdec} and their references.

By the dimension of a face $F$ of $\Delta$, we mean $|F|-1$ and the dimension of $\Delta$ is defined as
$\max\{\dim(F)|F\in \Delta\}$. Let $f_i$ be the number of $i$-dimensional faces of $\Delta$ (if $\Delta\neq \tohi$,
then $f_{-1}=1$), then $(f_{-1}, \ldots, f_{d-1})$ is called the \emph{$f$-vector} of $\Delta$, where
$d-1=\dim(\Delta)$. Now define $h_i$'s such that $ h(t)= \sum_{i=0}^d h_it^i= \sum_{i=0}^d f_{i-1}t^i(1-t)^{d-i}$. Then
$h(t)$ is called the \emph{$h$-polynomial} of $\Delta$. It can be shown that the Hilbert series of $S/I_\Delta$ is
$h(t)/(1-t)^ d$ (see \cite[Proposition 6.2.1]{hibi}). Denote by $\alpha(G)$ the \emph{independence number} of $G$, that
is, the maximum size of an independent set of $G$. Then the polynomial $I(G,x)=\sum_{i=0}^{\alpha(G)} a_ix^i$, where
$a_i$ is the number of independent sets of size $i$ in $G$, is called the \emph{independence polynomial} of $G$. Note
that $a_i=f_{i-1}$ where $(f_{-1},\ldots, f_{\alpha(G)-1})$ is the $f$-vector of $\Delta(\b G)$. There are many papers
related to this polynomial in the literature, see for example \cite{survey} and the references therein. It is easy to
check that the $h$-polynomial $h(t)$ of $\Delta(\b G)$ is $(1-t)^{\alpha(G)} I(G,t/(1-t))$.

A simplicial complex $\Delta$ is said to be \emph{Gorenstein*} when $S/I_\Delta$ is Gorenstein when $K=\Q$ is the field
of rational numbers (for the definition of Gorenstein rings and other algebraic notions the reader is referred to
\cite{CM ring}) and there is no vertex $v$ of $\Delta$ such that $\{v\}\cup F\in \Delta$ for every $F\in \Delta$. Note
that if $\Delta=\Delta(\b G)$, then $\Delta$ is Gorenstein* \ifof $G$ is Gorenstein over $\Q$ (that is $S/I(G)$ is
Gorenstein when $K=\Q$) and has no isolated vertex.

The Charney-Davis conjecture states that if $\Delta$ is a Gorenstein* flag complex of dimension $2e-1$, then
$(-1)^eh(-1) \geq 0$. In \cite[Problem 4]{frontier}, Richard P. Stanley mentioned this conjecture as one of the
``outstanding open problems in algebraic combinatorics'' at the start of the 21st century. This conjecture was proved
in dimension 3 in \cite{davis} and Stanley in \cite{stanley CD} showed that this conjecture holds for barycentric
subdivisions of shellable spheres. To see some other cases under which this conjecture is established, see \cite{froh,
welker}. The following is a ``more graph theoretical'' restatement of the Charney-Davis conjecture.

\begin{conj}[Charney \& Davis] \label{char conj}
If $G$ is a graph with no isolated vertices which is Gorenstein over $\Q$ and $\alpha(G)$ is even, then
$$(-1)^\f{\alpha(G)}{2} I(G,-\f{1}{2}) \geq 0.$$
\end{conj}

Next we recall some properties of Gorenstein graphs. A graph $G$ is called \emph{well-covered}, if all maximal
independent sets of $G$ have size $\alpha(G)$ and it is said to be a \emph{W$_2$ graph}, if $|\V(G)|\geq 2$ and every
pair of disjoint independent sets of $G$ are contained in two disjoint maximum independent sets. In some texts, W$_2$
graphs are called 1-well-covered graphs. The following lemma states the relation of Gorenstein graphs and W$_2$ graphs.
\begin{lem}[{\cite[Lemma 3.1]{large girth}}]\label{W2}
Every Gorenstein graph without isolated vertices is a W$_2$ graph.
\end{lem}

Recall that if $F\in \Delta$, then $\link_\Delta(F)=\{A\sm F| F\se A\in \Delta\}$. Suppose that $F\se \V(G)$. By
$\N[F]$ we mean $F\cup \{v\in \V(G)| uv\in \E(G)$ for some $u\in F\}$ and we set $G_F=G\sm \N[F]$. We simply write
$G_v$ instead of $G_{\{v\}}$.Thus if $F$ is independent, then $\link_{\Delta(\b G)} F= \Delta(\b{G_F})$. Another
combinatorial property of a Gorenstein* graph $G$ is that it has an \emph{Eulerian independence complex}, that is, $G$
is well-covered and $I(G_F,-1)=(-1)^{\alpha(G_F)}$ for every independent set $F$ of $G$ (one can readily check that
this condition is equivalent to $\Delta(\b G)$ being an Euler complex as defined in \cite[Definition 5.4.1]{CM ring}).
\begin{lem}\label{Euler}
\begin{enumerate}
\item \label{goren=>Euler} A graph without isolated vertices is Gorenstein (over $K$) \ifof it has an Eulerian
    independence
    complex
    and
    is CM (over $K$).
\item \label{Dehn} If $G$ has an Eulerian independence complex and $\alpha(G)$ is odd, then $I(G,-1/2)=0$.
\end{enumerate}
\end{lem}
\begin{proof}
Part \pref{goren=>Euler} is an especial case of \cite[Theorem 5.5.2]{CM ring}. For part \pref{Dehn}, note that if
$h(t)$ is the $h$-polynomial of $\Delta(\b G)$, then by the Dehn-Sommerville equation (\cite[Theorem 5.4.2]{CM ring})
we have $h(-1)=0$. But $h(-1)=2^{\alpha(G)}I(G,-1/2)$ and the result follows.
\end{proof}

Since every link of every CM simplicial complex is CM, one of the consequences of the above result is the following.
\begin{cor}\label{G_F goren}
Suppose that $G$ is a Gorenstein graph (over $K$), then for every non-maximal independent set $F$ of $G$, the graph
$G_F$ is also Gorenstein (over $K$).
\end{cor}

                    \section{Trung's construction and the Charney-Davis conjecture}

In \cite{pinter1} a method for constructing a W$_2$ graph from another W$_2$ graph is presented and it is shown that
all planar W$_2$ graphs with girth 4 are constructed by successively applying this method on a certain graph on 8
vertices. In \cite{large girth}, it is proved that all such graphs are indeed Gorenstein. Recently Trung has
generalized this construction, see \cite{planar goren}, and showed that this generalized construction preserves the
Gorenstein property. We recall this generalized construction (see \cite[Proposition 3.9]{planar goren}).

\begin{defn}
Suppose that $H$ is a graph and $v$ is a non-isolated vertex of $H$. Let $a$, $b$ and $c$ be three new vertices. Join
$c$ to $b$ and to every neighbor of $v$; join $b$ to $a$; and join $a$ to $v$. We denote the obtained graph by
$\T(H,v)$ and call this construction ``Trung's construction''.
\end{defn}

This construction is illustrated in \ref{fig-Tr}. Here we show that many properties of the independence complex of a
graph, is preserved by Trung's construction.

\begin{figure}[!ht]
\begin{center}
\includegraphics{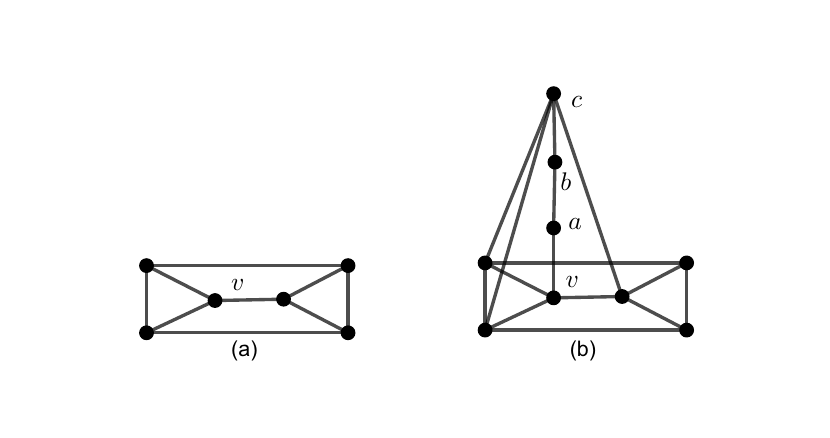}
\end{center}
\caption{(a) A graph $H$; (b) $\T(H,v)$} \label{fig-Tr}

\end{figure}
\begin{thm}\label{Tr goren}
Let $H$ be a graph and $v$ a non-isolated vertex of $H$. If $G=\T(H,v)$, then
\begin{enumerate}
\item \label{Tr Goren1} $\alpha(G)=\alpha(H)+1$;
\item \label{Tr Goren2} $G$ is Gorenstein (over $K$) \ifof $H$ is Gorenstein (over $K$).
\end{enumerate}
\end{thm}
\begin{proof}
\pref{Tr Goren1} is clear. \pref{Tr Goren2}: (\rgive) \cite[Proposition 3.9]{planar goren}; (\give)  Noting that
$H=G_b$, this follows from \pref{G_F goren}.
\end{proof}

\begin{thm}\label{Tr well}
Let $H$ be a graph and $v$ a non-isolated vertex of $H$. Then $G= \T(H,v)$ is well-covered \ifof $H$ is so.
\end{thm}
\begin{proof}
Let $F$ be a maximal independent set of $G$. We show that there is a maximal independent set of $H$ with $|F|-1$
vertices. Note that $|F\cap\{a,b,c\}|$ equals 1 or 2. In the latter case, $F\cap\{a,b,c\}=\{a,c\}$ and $(F\cap \V(H))
\cup\{v\}$ is a maximal independent set of $H$. Now suppose that $F\cap\{a,b,c\}=\{a\}$.  If $F\sm \{a\}$ is not a
maximal independent set of $H$, then $\{x\}\cup F\sm \{a\}$ is an independent set of $H$ for some $x\in \V(H)$. Since
$F\cup\{x\}$ is not independent in $G$, $x$ is adjacent to $a$, that is, $x=v$. This means that $N_H(v)\cap F=\tohi$.
Thus $F\cup\{c\}$ is an independent set of $G$ larger than $F$, a contradiction. Thus $F\sm\{a\}$ is a maximal
independent set of $H$. Similarly, in other cases that $|F\cap\{a,b,c\}|=1$, one can conclude that $F\cap \V(H)$ is a
maximal independent set of $H$. Consequently, cardinality of each maximal independent sets of $G$ is exactly one more
than the cardinality of some maximal independent set of $H$. Conversely each maximal independent set of $H$ can be
extended to a maximal independent set of $G$ with exactly one more vertex. From this the result follows.
\end{proof}

\begin{rem}\label{ind of Tr}
The argument in the proof of \pref{Tr well} shows that maximal independent sets of $\T(H,v)$ are exactly the sets of
the form $A\cup\{a\}$, $A\cup\{b\}$, $B\cup\{c\}$, $B\cup\{b\}$ or $B\cup \{a,c\}\sm \{v\}$, where $A, B$ are maximal
independent sets of $H$ with $v\in B\sm A$.
\end{rem}

\begin{thm}\label{Tr W2}
Let $H$ be a graph and $v$ a non-isolated vertex of $H$. Then $G= \T(H,v)$ is W$_2$ \ifof $H$ is so.
\end{thm}
\begin{proof}
(\give): Suppose that $A, B\in \Delta(\b H)$ are disjoint. We have to find disjoint $A', B'\in \Delta(\b H)$ such that
$|A'|=|B'|=\alpha(H)$, $A\se A'$ and $B\se B'$. As $G$ is W$_2$, there are disjoint independent sets $A'', B''$ of $G$
with size $\alpha(H)+1$ such that $A\se A''$ and $B\se B''$. If $|A''\cap \V(H)|= |B''\cap \V(H)|= \alpha(H)$, then we
are done. Thus according to the above remark, we can assume that $b\in B''$ and $a,c\in A''$. So $N_H(v)\cap
A''=\tohi$. If $v\notin B''$, then $A'=A''\cup\{v\}\sm \{a,c\}$ and $B'=B''\sm \{b\}$ have the required properties.
Hence we assume that $v\in B''$.

Set $A'''=A''\sm \{c\}$ and $B'''= B''\cup \{c\}\sm \{b\}$. Since $G$ is W$_2$, we can extend $A'''$ and $B'''$ to
disjoint maximum size independent sets of $G$. Equivalently, there is a $u\in \V(G)\sm (B'''\cup A''')$ such that
$A'''\cup\{u\}$ is independent. Note that $u\neq a,b,c$ and hence $u\in \V(H)\sm B''$. Therefore, $A'=A''\cup
\{u\}\sm\{a,c\}$ and $B'=B''\sm\{b\}$ are disjoint maximum size independent sets of $H$ containing $A$ and $B$
respectively, as required.

(\rgive): Let $A, B$ be disjoint independent sets of $G$. We must find disjoint maximum size independent sets $A', B'$
of $G$ such that $A\se A', B\se B'$. We consider several cases:

\subparagraph{\emph{Case 1:}} $c\notin A$ and $v\in B$. Then $A_0=A\cap H$ and $B_0=B\cap H$ are disjoint independents
sets of $H$ and we can extend them to disjoint maximal independent sets $A'_0$ and $B'_0$ of $H$, respectively. If
$A\cap \{a,b\}$ is nonempty, then let $E=A\cap \{a,b\}$ and if $A\cap \{a,b\}= \tohi$, let $E=\{a\}$. Now set
$A'=A'_0\cup E$. If $b\in B$, let $B'=B'_0\cup \{b\}$, else let $B'=B'_0\cup \{c\}$. One can readily check that $A'$
and $ B'$ satisfy the required conditions.

\subparagraph{\emph{Case 2:}} $c\in A$ and $v\in B$. Let $A_0=A \cap H_v$ and $B_0=B\cap H_v$. Note that by
\cite[Theorem 3]{pinter2}, $H_v$ is W$_2$ with $\alpha(H_v)=\alpha(H)-1$. Suppose that $A'_0$ and $B'_0$ are disjoint
maximal independent extensions of $A_0$ and $B_0$ in $H_v$. Now $A'=A'_0\cup\{c,a\}$ and $B'=B'_0\cup\{v,b\}$ are
disjoint maximum size independent sets of $G$. Since $c\in A$ and $v\in B$, we have $A\cap N_H(v)=\tohi$ and $B\cap
N_H(v)= \tohi$. Therefore, $A\cap H=A_0$ and $B\cap H= B_0\cup\{v\}$ and it follows that $A\se A'$ and $B\se B'$, as
required.

Note that if $v\in A$, then by changing the names of $A$ and $B$, case 1 or 2 occurs. So we can assume that
$v\notin A\cup B$.

\subparagraph{\emph{Case 3:}} $v\notin A\cup B$ and $a,c\in A$. Let $A_0=(A \cap H)\cup \{v\}$ ($A_0$ is independent
because $c\in A$ and hence $\N_H(v)\cap A=\tohi$)  and $B_0=B\cap H$ and extend them to disjoint maximum independent
sets $A'_0$ and $B'_0$ of $H$. Now $A'=(A'_0\sm \{v\}) \cup \{a,c\}$ and $B'=B'_0 \cup \{b\}$ have the required
properties.

\subparagraph{\emph{Case 4:}} $v\notin A\cup B$, $c\in A$ and $a\notin A$. Let $A_0=(A \cap H)\cup \{v\}$ and
$B_0=B\cap H$ and extend them to disjoint maximum independent sets $A'_0$ and $B'_0$ of $H$. Set $A'=A'_0\cup\{c\}$. If
$b\in B$, set $B'=B'_0 \cup \{b\}$ and if $b\notin B$ set $B'=B'_0\cup\{a\}$.

\subparagraph{\emph{Case 5:}} $v,c\notin A\cup B$ and $a\in A$. Let $A_0=A \cap H$ and $B_0=B\cap H$ and extend them to
disjoint maximum independent sets $A'_0$ and $B'_0$ of $H$. Set $B'=B'_0\cup\{b\}$ and if $v\in A'_0$, set $A'=(A'_0\sm
\{v\})\cup \{a,c\}$, else set $A'=A'_0\cup\{a\}$.

\subparagraph{\emph{Case 6:}} $v,a,c\notin A\cup B$ and $b\in B$.  Let $A_0=A$ and $B_0=B\cap H$ and extend them to
disjoint maximum independent sets $A'_0$ and $B'_0$ of $H$. Set $B'=B'_0 \cup\{b\}$ and if $v\in A'_0$, let
$A'=A'_0\cup\{c\}$ and if $v\notin A'_0$, set $A'= A'_0\cup\{a\}$.

\subparagraph{\emph{Case 7:}} $v,a,b,c\notin A\cup B$. Let $A'_0$ and $B'_0$ be disjoint maximum size independent sets
of $H$ containing $A$ and $B$, respectively. Then $v$ is not in at least one of $A'_0$ or $B'_0$, say $v\notin A'_0$.
Then $A'=A'_0\cup \{a\}$ and $B'= B'_0\cup\{b\}$ have the required properties.
\end{proof}

In the next theorem we present a formula for computing the independence polynomial of $\T(H,v)$ in terms of
independence polynomials of $H$ and $H_v$.

\begin{thm}\label{Tr-ind poly}
Let $H$ be a graph and $v$ be a non-isolated vertex of $H$. Then
$$I(\T(H,v),x)=(2x+1)I(H,x)+(x+x^2)I(H_v,x).$$
\end{thm}
\begin{proof}
Throughout the proof, $F$ always denotes an independent set of $G=\T(H,v)$ with $|F|=i$. We denote $F\cap\{a,b,c\}$ by
$F_0$. Also for any graph $\Gamma$ by $a_i(\Gamma)$ we mean the number of independent sets of $\Gamma$ with cardinality
$i$. If $i<0$, we set $a_i(\Gamma)=0$. Note that $F_0=\tohi$ \ifof $F$ is an independent set of $H$ with size $i$. Thus
there are $a_i(H)$ such $F$'s. Also $F_0=\{a\}$ \ifof $F=F_1\cup\{a\}$ for an independent set $F_1$ of $H-v$ with
$|F_1|=i-1$. Thus there are $a_{i-1}(H-v)$ choices of $F$ with $F_0=\{a\}$. Similarly, there are $a_{i-1}(H)$
 choices of $F$ with $F_0=\{b\}$.

Now assume that $F_0=\{c\}$. If $v\in F$, then $F\sm \{v\}$ is an independent set of $H_v$  with cardinality $i-2$
and  conversely by adding $v$ and $c$ to any such independent set of $H_v$, we get an $F$ with $F_0=\{c\}$ and
$v\in F$. Similarly, those $F$ with $F_0=\{c\}$ and $v\notin F$ correspond to the independent sets of $H_v$ with
size $i-1$ (note that as $c\in F$, we have $N_H(v)\cap F=\tohi$). Therefore, there are totally
$a_{i-1}(H_v)+a_{i-2}(H_v)$ choices for $F$ with $F_0=\{c\}$.

Finally, if $F_0=\{a,c\}$, then $F\cap \N_{H}[v]=\tohi$ and hence $F\cap H\se H_v$. Consequently, there is a
one-to-one correspondence between those $F$ with $F_0=\{a,c\}$ and independent sets of $H_v$ with size $i-2$. So
there are $a_{i-2}(H_v)$ choices for $F$ with $F_0=\{a,c\}$.

Totally, we get that $a_i(G)=a_i(H)+a_{i-1}(H-v)+a_{i-1}(H)+a_{i-1}(H_v)+2a_{i-2}(H_v)$. Note that
$a_{i-1}(H-v)+a_{i-2}(H_v)= a_{i-1}(H)$, because $a_{i-1}(H-v)$ is number of independent sets of $H$ with
cardinality $i-1$ which do not contain $v$ and $a_{i-2}(H_v)$ is the number of independent sets of $H$ with size
$i-1$ which contain $v$. We conclude that
$$a_i(G)=a_i(H)+ 2a_{i-1}(H) +a_{i-1}(H_v)+a_{i-2}(H_v).$$
Multiplying by $x^i$ and taking summation over $i=0, \ldots, \alpha(G)$ we get the desired equation.
\end{proof}

\begin{cor}\label{Tr charney}
Let $H$ be a graph without isolated vertices which is Gorenstein over $\Q$ such that $\alpha(H)$ is odd and assume that
$v \in\V(H)$. Then $G=\T(H,v)$ satisfies the Charney-Davis conjecture \ifof $H_v$ does so.
\end{cor}
\begin{proof}
According to \pref{Tr-ind poly}, $I(G,-1/2)= (-1/4)I(H_v,-1/2)$. Noting that $\alpha(H_v)=\alpha(G)-2$, we conclude
that $(-1)^{\alpha(G)/2}I(G,-1/2)\geq 0$ \ifof $(-1)^{\alpha(H_v)/2}I(H_v,-1/2)\geq 0$, as claimed.
\end{proof}

Recall that the Stanley-Reisner algebra of the disjoint union of two graphs is isomorphic to the tensor product of the
Stanley-Reisner algebras of the two graphs. Hence a graph is Gorenstein over $K$ \ifof all of its connected components
are Gorenstein over $K$.
\begin{thm}\label{plan charney}
Suppose that $G$ is a planar Gorenstein graph without isolated vertices, girth$(G)\geq 4$ and $\alpha(G)$ is even.
Then the Charney-Davis conjecture holds for $G$.
\end{thm}
\begin{proof}
We prove the statement by induction on $|\V(G)|$. If $G$ is not connected, say $G$ is a disjoint union of $G_1$ and
$G_2$, then both $G_1$ and $G_2$ are Gorenstein graphs without isolated vertices. If $\alpha(G_1)$ and $\alpha(G_2)$
are odd, then according to \pref{Euler}, $I(G_i,-1/2)=0$ for both $i$'s and if $\alpha(G_1)$ and $\alpha(G_2)$ are
even, then by the induction hypothesis, $(-1)^{\alpha(G_i)/2} I(G_i,-1/2)\geq 0$ for both $i=1,2$. Therefore the result
follows from the fact that $I(G,x)=I(G_1,x)I(G_2,x)$ (see for example \cite[Section 2]{survey}).

Thus we assume that $G$ is connected. If $G$ has girth $\geq 5$, then as $G$ is W$_2$ and by \cite[Theorem 7]{pinter2},
$G \cong K_2$ or $G \cong C_5$, both of which satisfy the Charney-Davis conjecture. So we can suppose that
girth$(G)=4$. Then according to \cite[Lemma 3.2]{large girth}, $G$ is constructed by several application of Trung's
construction on $C_5$, where in each application the chosen vertex should be a vertex of degree 2. Thus we can assume
that $G=\T(H,v)$, for a planar graph $H$ of girth at least 4 which does not have any isolated vertex and $\deg_H(v)=2$.
Clearly $H_v$ is planar and has girth at least 4. Also it is Gorenstein by \pref{G_F goren}. If $H_v$ has an isolated
vertex, say $y$, then in $H$, $y$ is adjacent to both neighbors of $v$ (else $\deg_H(y)=1$, which contradicts
\pref{W2}). Consequently, $\{y,v\}\cup N_H(v)$ is a 4-cycle in $H$. But this is against \cite[Theorem 2]{pinter1},
which says that every vertex on a 4-cycle in a W$_2$ graph has degree at least 3. Hence $H_v$ has no isolated vertex
and the Charney-Davis conjecture holds for $H_v$ by the induction hypothesis. Now the result follows from \pref{Tr
charney}.
\end{proof}



\begin{thebibliography}{99}

%
\bibitem{welker} F. Brenti and V. Welker, $f$-vectors of barycentric subdivisions, \emph{Math. Z.} \textbf{259},
    849--865, 2008.

%
\bibitem{CM ring} W. Bruns and J\"urgen Herzog, \emph{Cohen-Macaulay Rings}, Cambridge University Press,
    Cambridge, 1993.

\bibitem{charney} R. Charney and M. Davis, Euler characteristic of a nonpositively curved, piecewise Euclidean
    manifold, \emph{Pac. J. Math.} \textbf{171}, 117--137, 1995.

%
%
%
\bibitem{davis} M. Davis and B. Okun, Vanishing theorems and conjectures for the $l^2$-homology of
    right-angled Coxeter groups, \emph{Geom. Topol.} \textbf{5}, 7--74, 2001.

\bibitem{froh} A. Frohmader, The Charney-Davis conjecture for certain subdivisions of spheres, preprint,
    arXiv:0806.4213,
    2008.
%
%
\bibitem{hibi} J. Herzog  and T. Hibi, \textit{Monomial Ideals}, Springer-Verlag, London , 2011.

\bibitem{large girth} D. T. Hoang, N. C. Minh and T. N. Trung, Cohen-Macaulay graphs with large girth, \emph{J.
    Algebra
    Appl.}
    \textbf{14}:7,
    paper No. 1550112, 2015.

%
\bibitem{survey} V.E. Levit and E. Mandrescu, The independence polynomial of a graph --- a survey,
    \emph{Proceedings  of the 1st International Conference on Algebraic Informatics}, 233--254,
    Aristotle Univ. Thessaloniki, Thessaloniki, 2005.

\bibitem{my vdec} A. Nikseresht, Chordality of clutters with vertex decomposable dual and ascent of clutters,
    \emph{J. Combin. Theory ---Ser. A}, accepted for publication,  arXiv: 1708.07372.

\bibitem{our chordal} A. Nikseresht and R. Zaare-Nahandi, On generalization of cycles and chordality to clutters from
    an algebraic viewpoint, \emph{Algebra Colloq.} \textbf{24}:4, 611--624,  2017.

\bibitem{pinter1} M. R. Pinter, A class of planar well-covered graphs with girth four, \emph{J. Graph Theory}
    \textbf{19}:1, 69--81, 1995.

\bibitem{pinter2} ---, A class of well-covered graphs with girth four, \emph{Ars Combin.} \textbf{45}, 241--255,
    1997.

%
\bibitem{stanley} R. P. Stanley, \emph{Combinatorics and Commutative Algebra}, Birkh\"auser, Boston, 1996.

\bibitem{stanley CD} ---, Flag $f$-vectors and the cd-index, \emph{Math. Z.} \textbf{216}, 483--499, 1994.

\bibitem{frontier} ---, Positivity problems and conjectures in algebraic combinatorics, in: V. Arnold,
    et al (Eds.), \emph{Mathematics: Frontiers and Prospectives}, American Mathematical Society, pp. 295--320, 2000.
%
%
\bibitem{planar goren} T. N. Trung, A characterization of Gorenstein planar graphs, in T. Hibi, ed., \emph{Adv. Stud.
    Pure Math.}, Vol. 77: \emph{The 50th Anniversary of Gr\"obner Bases}, 399--409, 2018, arXiv:1603.00326v2.


\end{thebibliography}
\end{document}